\documentclass[12pt]{amsart}
\usepackage{amsmath,amssymb,amsfonts,amsthm,amsopn}
\usepackage{graphics}
\usepackage{xcolor}
\usepackage{pb-diagram}
\usepackage[title]{appendix}

\newtheorem{theorem}{Theorem}[section]

\newtheorem{corollary}[theorem]{Corollary}
\newtheorem{proposition}[theorem]{Proposition}
\newtheorem{definition}[theorem]{Definition}

\newtheorem{example}[theorem]{Example}
\newtheorem{remark}[theorem]{Remark}

\newcommand{\beqa}{\begin{eqnarray*}}
\newcommand{\eeqa}{\end{eqnarray*}}

\DeclareMathOperator*{\essupp}{ess\,sup\,}

\newcommand{\field}[1]{\mathbb{#1}}
\newcommand{\bR}{\field{R}}        
\newcommand{\bC}{\field{C}}        
\def\la{\lambda}

\def\cF{\mathcal{F}}              
\def\cS{\mathcal{S}}
\def\cD{\mathcal{D}}

\def\cB{\mathcal{B}}
\def\cE{\mathcal{E}}
\def\cG{\mathcal{G}}
\def\cM{\mathcal{M}}

\def\cA{\mathcal{A}}

\def\rd{\bR^d}

\def\rdd{{\bR^{2d}}}

\def\intrd{\int_{\rd}}

\def\<{\left<}
\def\>{\right>}

\def\mv1{M_v^1}

\def\phas{(x,\xi )}
\def\mn{(m,n)}
\def\mn'{(m',n')}

\newcommand{\norm}[1]{\lVert#1\rVert}

\def\Ren{\mathbb{R}^d}

\def\Sn2{S_{2}(L^{2}(\Ren))}
\def\S1{S_{1}(L^{2}(\Ren))}
\def\sig00{\sigma_{0,0}}

\def\la{\langle}
\def\ra{\rangle}

\begin{document}
\begin{abstract} 
	Metaplectic Wigner distributions generalize the most popular time-frequency representations, such as the short-time Fourier transform (STFT) and $\tau$-Wigner distributions, using metaplectic operators. However, in order for a metaplectic Wigner distribution to measure local time-frequency concentration of signals, the additional property of shift-invertibility is fundamental. In addition, metaplectic atoms provide different ways to model signals. Namely, signals can be written as discrete superpositions of these operators, providing original ways to represent signals, with applications to machine learning, signal analysis, theory of pseudodifferential operators, to mention a few. Among all shift-invertible distributions, Wigner-decomposable metaplectic Wigner distributions provide the most straightforward generalization of the STFT. In this work, we focus on metaplectic atoms of Wigner-decomposable shift-invertible metaplectic distributions and characterize the associated metaplectic Gabor frames. 
\end{abstract}

\title[Metaplectic Gabor frames of Wigner-decomposable distributions]{Metaplectic Gabor frames of Wigner-decomposable distributions}

\author{Elena Cordero}
\address{Universit\`a di Torino, Dipartimento di Matematica, via Carlo Alberto 10, 10123 Torino, Italy}
\email{elena.cordero@unito.it}
\author{Gianluca Giacchi}
\address{Università di Bologna, Dipartimento di Matematica,  Piazza di Porta San Donato 5, 40126 Bologna, Italy; Institute of Systems Engineering, School of Engineering, HES-SO Valais-Wallis, Rue de l'Industrie 21, 1950 Sion, Switzerland; Lausanne University Hospital and University of Lausanne, Lausanne, Department of Diagnostic and Interventional Radiology, Rue du Bugnon 46, Lausanne 1011, Switzerland. The Sense Innovation and Research Center, Avenue de Provence 82
1007, Lausanne and Ch. de l’Agasse 5, 1950 Sion, Switzerland.}
\email{gianluca.giacchi2@unibo.it}

\thanks{}
\subjclass[2010]{42B35,42A38}

\keywords{Time-frequency analysis, modulation spaces, Wiener amalgam spaces, time-frequency representations, metaplectic group, symplectic group}
\maketitle

\section{Introduction}
Decomposing finite-energy signals as superpositions of fundamental functions (\textit{atoms}), with preservation of the energy content, is one of the main issue in signal analysis. If $\cB=\{\varphi_k\}_{k}\subseteq L^2(\rd)$ is a basis of $L^2(\rd)$, then every signal $f\in L^2(\rd)$ can be  decomposed as 
\[
	f = \sum_k\la f,\varphi_k\ra \varphi_k,
\]
and $\norm{f}_2=(\sum_k|\la f,\varphi_k\ra|^2)^{1/2}$. This decomposition is unique, because $\cB$ is a basis, but in the applications it may be useful to have different representations for the same signal, and this is pursued using frames. In time-frequency analysis, where the time and frequency behaviours of signals are treated simultaneously, \textit{Gabor frames} have become popular. For a function $g\in L^2(\rd)$ and $(x,\xi)\in\rdd$, the associated time-frequency shift is defined as $\pi(x,\xi)g(t)=e^{2\pi i\xi\cdot t}g(t-x)$, that is,  the composition of the  \textit{translation} $T_x g(t)=g(t-x)$ and \textit{modulation} $M_\xi g(t)=e^{2\pi i \xi\cdot t}f(t)$ operators, $x,\xi\in\rd$, see Section \ref{sec:preliminaries} for details. A \textit{Gabor frame} is a family $\mathcal{G}(g,\Lambda)=\{\pi(\lambda)g\}_{\lambda\in\Lambda}$, where $g\in L^2(\rd)\setminus\{0\}$ and $\Lambda\subseteq\rdd$ is a  discrete set (often a lattice, i.e., a discrete subgroup of $\rdd$), with the following energy-preservation property: there exist $A,B>0$ such that 
\begin{equation}\label{GF}
	A\norm{f}_2^2\leq\sum_{\lambda\in\Lambda}|\la f,\pi(\lambda)f\ra|^2\leq B\norm{f}_2^2, \qquad f\in L^2(\rd).
\end{equation}
If $\cG(g,\Lambda)$ is a Gabor frame, then there exists a window $\gamma=\gamma(g)\in L^2(\rd)$ such that
\begin{equation}\label{Decompintro}
	f=\sum_{\lambda\in\Lambda}\la f,\pi(\lambda)g\ra\pi(\lambda)\gamma, \qquad f\in L^2(\rd),
\end{equation}
with unconditional convergence in the norm of $L^2(\rd)$. Stated differently, Gabor frames allow to decompose signals into discrete superpositions of time-frequency shifts, which take over the role of building blocks for finite-energy signals. The coefficients in (\ref{Decompintro}) define the \textit{short-time Fourier transform} (STFT). Namely, if $g\in L^2(\rd)\setminus\{0\}$ is fixed and $f\in L^2(\rd)$, the STFT of $f$ with respect to the window $g$ is the function defined as:
\begin{equation}\label{STFT}
	V_gf(x,\xi)=\la f, \pi(x,\xi)g\ra=\cF_2\mathfrak{T}_{ST}(f\otimes\bar g)(x,\xi), \qquad (x,\xi)\in\rdd,
\end{equation}
where $\cF_2$ is the partial Fourier transform with respect to the second variable and $\mathfrak{T}_{ST}$ is a rescaling operator, see Section \ref{sec:preliminaries} below. The STFT is a joint time-frequency representation of $f$, it describes the local time-frequency behaviour of $f$ and it allows to recover both $f$ and its Fourier transform, via inversion formulae. However, in many contexts, the STFT is not the best time-frequency representation in terms of mathematical properties, and it is preferred to appeal to different distributions, such as the $\tau$-Wigner distributions \cite{bogetal}, defined as
\begin{equation}\label{tauWigner}
	W_\tau(f,g)(x,\xi)=\int_{\rd}f(x+\tau t)\overline{g(x-(1-\tau)t)}e^{-2\pi i\xi\cdot t}dt, \quad f,g\in L^2(\rd),
\end{equation}
for  $\phas\in\rdd$, $\tau\in\bR$. As for the STFT, these distributions can be written as composition of operators. Namely,
\[
	W_\tau(f,g)(x,\xi)=\cF_2\mathfrak{T}_\tau (f\otimes\bar g)(x,\xi), \qquad f,g\in L^2(\rd),\,\quad \phas\in\rdd,
\]
where $\mathfrak{T}_\tau$ is a suitable rescaling operator, see Section \ref{sec:preliminaries} in the sequel. However, this is not the only similarity that $\tau$-Wigner distributions share with the STFT. In fact, it was proved in \cite{CG2023frames} that the $\tau$-Wigner distributions ($\tau\neq0,1$) can be rewritten in terms of $L^2$ inner products as
\begin{equation}\label{tau-atoms}
	W_\tau(f,g)(x,\xi)=\la f,\pi_\tau(x,\xi)g\ra, \qquad f,g\in L^2(\rd),\,\quad \phas\in\rdd,
\end{equation}
where $\pi_\tau(x,\xi)$ is, up to  rescaling and a chirp phase factor, the composition of a rescaled time-frequency shift and a unitary change of variables. 

Another issue is the frame property. Gabor frames are enormously popular in applications such that  imaging, radar, audio processing
and quantum mechanics \cite{book}. Though, in some framework it is  more useful to employ \emph{generalized versions} of them, see for example \cite{DGosson}.  In our context, let us first consider the atoms generating from the $\tau$-Wigner distributions, defined in \eqref{tau-atoms} above.   Given $g\in L^2(\rd)\setminus\{0\}$ and $\Lambda\subseteq\rdd$ discrete, if the family $\cG_\tau(g,\Lambda)=\{\pi_\tau(\lambda)g\}_{\lambda\in\Lambda}$ defines a frame for $L^2(\rd)$, then any signal $f\in L^2(\rd)$ can be decomposed as 
\[
	f=\sum_{\lambda\in\Lambda}\la f,\pi_\tau(\lambda)g\ra\pi_\tau(\lambda)\gamma_\tau,
\]
for a suitable $\gamma_\tau=\gamma_\tau(g,\tau)\in L^2(\rd)$, providing a different way to represent signals in terms of the \textit{atoms} $\pi_\tau(\lambda)g$. For many aspects, $\tau$-Wigner distributions represent signals better than the STFT, but they are only the tip of the iceberg of a wide variety of time-frequency representations, called \textit{metaplectic Wigner distributions}, introduced and studied in \cite{CGshiftinvertible,CG2023frames,CGR2022,CR2022p1,CR2022,Giacchi}. Among all the metaplectic Wigner distributions, the Wigner-decomposable representations are the immediate generalization of the STFT and $\tau$-Wigner distributions. In fact, they are defined, up to a chirp phase factor, in terms of the Fourier transform with respect to the second variable and a rescaling operator:
\[
	W_\cA(f,g)=\cF_2\mathfrak{T}_E(f\otimes\bar g), \qquad f,g\in L^2(\rd),
\]
where $E\in GL(2d,\bR)$, $\mathfrak{T}_E F=|\det(E)|^{1/2}F(E\cdot)$ and $\cA$ denotes a symplectic matrix which is related to $\cF_2$ and $\mathfrak{T}_E$, see Section \ref{sec:MAWD}. 

In this work, we derive the expression of the \textit{metaplectic atoms} associated to Wigner-decomposable distributions and showcase the related frames. 
Moreover, we characterize modulation spaces for Wigner-decomposable distributions, extending the work \cite{BCGT2020,CT2020}.

\textbf{Overview.} Section 2 contains definitions and preliminaries. Section 3 introduces the main protagonists of this theory: the metaplectic Wigner distributions and related atoms, showing their fundamental properties. Section 4 is the core of this study: (totally) Wigner decomposable distributions are defined  and the related atoms are computed explicitly. Also, the inverse of a Wigner decomposable atom is shown and relations among Wigner decomposable and totally Wigner decomposable atoms are highlighted. The last section is devoted to frame theory for (totally) Wigner decomposable atoms. In particular, Theorem \ref{main} shows the relation between these frames and the classical Gabor ones.  Finally, modulation and Wiener amalgam spaces  can be defined by means of Wigner decomposable distributions and related frames, cf. Theorems \ref{main2} and \ref{main3}.

\section{Preliminaries and Notation}\label{sec:preliminaries}
Throughout this work, $xy=x\cdot y$ denotes the standard inner product on $\rd$. If $0<p<\infty$ and $f:\rd\to\bC$ is measurable, $\norm{f}_p=(\int_{\rd}|f(x)|^pdx)^{1/p}$, whereas $\norm{f}_\infty:=\essupp_{x\in\rd}|f(x)|$. These are norms if $p\geq1$ and quasi-norms if $0<p<1$. $\cS(\rd)$ denotes the spaces of Schwartz functions, whereas $\cS'(\rd)$ denotes its topological dual, the space of tempered distributions. We denote with $\la f,g\ra=\int f(t)\overline{g(t)}dt$ the inner product on $L^2(\rd)$ (antilinear in the second component). We use the same notation to denote the duality pairing of $\cS'(\rd)\times\cS(\rd)$. Namely, if $f\in\cS'(\rd)$ and $\varphi\in\cS(\rd)$,
\[
	\la f,\varphi\ra :=f(\bar\varphi).
\] 
For $z=(x,\xi)\in\rdd$, the \textit{time-frequency shift} $\pi(z)$ is the operator
\[
	\pi(z)f(t)=\pi(x,\xi)f(t)=M_\xi T_x f(t)=e^{2\pi i\xi\cdot t}f(t-x), \qquad f\in L^2(\rd),
\]
where $M_\xi f(t)=e^{2\pi i\xi\cdot t}f(t)$ and $T_xf(t)=f(t-x)$ are the \textit{modulation} and the \textit{translation} operators, respectively.\\

If $f,g:\rd\to\bR$, $f\lesssim g$ means that there exists $C>0$ such that $f(t)\leq Cg(t)$ for all $t\in\bR^d$. If $f\lesssim g\lesssim f$, we write $f\asymp g$. We denote by $v$ a continuous, positive, even, submultiplicative weight function on $\rdd$, i.e., 
$ v(z_1+z_2)\leq v(z_1)v(z_2)$, for all $ z_1,z_2\in\rdd$. 
We say that $w\in \mathcal{M}_v(\rdd)$ if $w$ is a positive, continuous, even weight function  on $\rdd$  {\it
	$v$-moderate}:
$ w(z_1+z_2)\lesssim v(z_1)w(z_2)$  for all $z_1,z_2\in\rdd$. Fundamental examples are the polynomial weights
\begin{equation}\label{vs}
	v_s(z) =(1+|z|)^{s},\quad s\in\bR,\quad z\in\rdd.
\end{equation}
For $0<p,q\leq\infty$, $m\in\cM_v(\rdd)$ and $F=F(x,y):\rdd\to\bC$ measurable, we denote by
\[
	\norm{F}_{L^{p,q}_m}:=\norm{y\mapsto \norm{m(\cdot,y)F(\cdot,y)}_p }_q.
\]
If $p,q<\infty$,
\[
	\norm{F}_{L^{p,q}_m}=\left(\int_{\rd} \left(\int_{\rd} |F(x,y)|^pm(x,y)^p dx \right)^{q/p} dy\right)^{1/q}.
\]
(Obvious changes for $p=\infty$ or $q=\infty$). 
The mixed-norm Lebesgue space $L^{p,q}_m(\rdd)$ is the space of measurable functions $F:\rdd\to\bC$ such that $\norm{F}_{L^{p,q}_m}<\infty$.\\

If $f,g:\rdd\to\bC$, their tensor product is denoted by $f\otimes g(x,y)=f(x)g(y)$. If $f,g\in \cS'(\rd)$, their tensor product is the unique tempered distribution $f\otimes g\in\cS'(\rdd)$ characterized by its action on tensor products $\varphi\otimes\psi\in \cS(\rdd)$ by:
\[
	\la f\otimes g,\varphi\otimes \psi\ra =\la f,\varphi\ra\la g,\psi\ra.
\]

\subsection{Fourier transform}
Let $f\in\cS(\rd)$, the \textbf{Fourier transform} of $f$, denoted by $\hat f$, is the function defined as
\[
	\hat f(\xi)=\int_{\rd}f(x)e^{-2\pi i\xi\cdot x}dx, \qquad \xi\in\rd.
\]
The Fourier transform of $f\in\cS'(\rd)$ is defined by duality as
\[
	\la \hat f,\hat \varphi\ra = \la f,\varphi\ra, \qquad \varphi\in\cS(\rd).
\]
The Fourier transform operator will be denoted with $\cF$. It is topological isomorphism of $\cS(\rd)$ and $\cS'(\rd)$ as well as a unitary operator on $L^2(\rd)$. For $F=F(x,y)\in \cS(\rdd)$, we set
\[
	\cF_2 F(x,\xi):=\int_{\rd} F(x,y)e^{-2\pi i\xi\cdot y}dy
\]
the partial Fourier transform \textit{with respect to the second variables}. It is a topological isomorphism of $\cS(\rdd)$ to itself that extends to a topological isomorphism of $\cS'(\rdd)$. This extension is characterized by its action on tensor products $\varphi\otimes\psi\in \cS(\rdd)$ by:
\[
	\la \cF_2(f\otimes g),\varphi\otimes\psi\ra =\la f,\varphi\ra\la \hat g,\psi\ra=\la f,\varphi\ra\la g,\cF^{-1}\psi\ra=\la f\otimes g,\cF_2^{-1}(\varphi\otimes \psi)\ra.
\]
$\cF_2$ also extends to a unitary operator on $L^2(\rdd)$.

\subsection{Modulation spaces \cite{KB2020,Elena-book,F1,book,Galperin2004}}
Let $f\in \cS'(\rd)$ and fix $g\in \cS(\rd)\setminus\{0\}$. The \textit{short-time Fourier transform} of $f$ with respect to $g$ is the function defined in \eqref{STFT}.
It is possible to recover a signal $f\in\cS'(\rd)$ as an integral superposition of time-frequency shifts: for fixed $g,\gamma\in\cS(\rd)$ such that $\la g,\gamma\ra \neq 0$,
\begin{equation}\label{invSTFT}
	f = \frac{1}{\la \gamma,g \ra}\int_{\rdd}V_gf(x,\xi)\pi(x,\xi)\gamma dxd\xi,
\end{equation}
where the integral must be intended in the weak sense of vector-valued integration.

Other time-frequency representation that we will consider  are the (cross-)\textit{$\tau$-Wigner distributions},  defined in \eqref{tauWigner}.

The cases $\tau=0,1$ are known as \textit{Rihacek} and \textit{conjugate-Rihacek} distributions.

For $0<p,q\leq\infty$ and $m\in\cM_v(\rdd)$, the \textbf{modulation space} $M^{p,q}_m(\rd)$ is the space of tempered distributions $f\in\cS'(\rd)$ such that
\[
	\norm{f}_{M^{p,q}_m}=\norm{V_gf}_{L^{p,q}_m}<\infty.
\]
If $p=q$, we write $M^p_m(\rd)=M^{p,p}_m(\rd)$. $\norm{\cdot}_{M^{p,q}_m}$ are quasi-norms (norms if $p,q\geq1$) and different choices of $g$ give rise to equivalent (quasi-)norms. We recall the basic inclusion and duality properties of these spaces: if $0<p_1\leq p_2\leq\infty$, $0< q_1\leq q_2\leq\infty$ and $m_1,m_2\in\cM_{v_s}(\rdd)$ are such that  $m_2\lesssim m_1$, then
$
	\cS(\rd)\hookrightarrow M^{p_1,q_1}_{m_1}(\rd)\hookrightarrow M^{p_2,q_2}_{m_2}(\rd)\hookrightarrow \cS'(\rd).
$
Moreover, if $1\leq p,q<\infty$, $(M^{p,q}_m(\rd))'=M^{p',q'}_{1/m}(\rd)$. 

For fixed $g\in \cS(\rd)\setminus\{0\}$, $0<p,q\leq\infty$ and $m_1, m_2\in \cM_v(\rd)$, the \textbf{Wiener amalgam space} $W(\cF L^p_{m_1}, L^q_{m_2})(\rd)$ is defined in terms of the STFT as the space of tempered distributions $f\in\cS'(\rd)$ such that the (quasi-)norm
\[
	\norm{f}_{W(\cF L^p_{m_1},L^q_{m_2})}:=\left(\int_{\rd}\left(\int_{\rd}|V_gf(x,\xi)|^pm_1(\xi)^pd\xi\right)^{q/p}m_2(x)^qdx\right)^{1/q},
\]
is finite, with the obvious adjustments for $\max\{p,q\}=\infty$. Again, different choices of the window $g$ yield to equivalent (quasi-)norms. This definition highlights that $W(\cF L^p_{m_1},L^q_{m_2})(\rd)=\cF M^{p,q}_{m_1\otimes m_2}(\rd)$.

\subsection{Gabor frames}
Let $g\in L^2(\rd)$ and $\Lambda\subseteq\rdd$ be a discrete set. The \textbf{Gabor system} $\cG(g,\Lambda)$ is defined as the set of the time-frequency shifts of $g$ parametrized by $\Lambda$. Namely,
\[
	\cG(g,\Lambda)=\{\pi(\lambda)g\}_{\lambda\in\Lambda}.
\]
The Gabor system $\cG(g,\Lambda)$ defines a Gabor frame of $L^2(\rd)$ if there exist $A,B>0$ such that \eqref{GF} holds true.
Observe that this is equivalent to 
\[
	A\norm{f}_2^2\leq \sum_{\lambda\in\Lambda}|V_gf(\lambda)\ra|^2\leq B\norm{f}_2^2, \qquad f\in L^2(\rd).
\]

If $\cG(g,\Lambda)$ is a Gabor frame, then 
\[
	f=\sum_{\lambda\in \Lambda}\la f,\pi(\lambda)g\ra \pi(\lambda)\gamma,
\]
where $\gamma\in L^2(\rd)$ is the so-called \textit{dual window}, which depends on $g$ and $\Lambda$, and the series converges unconditionally in the $L^2(\rd)$-norm.

\subsection{Symplectic group and metaplectic operators}
For details we refer to \cite{Gos11}. Let $I_{d\times d}$ and $0_{d\times d}$ denote the identity matrix and the matrix with all zero entries, respectively. Let us denote with $J$ the standard symplectic matrix:
\[
	J=\begin{pmatrix}
		0_{d\times d} & I_{d\times d}\\
		-I_{d\times d} & 0_{d\times d}
	\end{pmatrix}.
\]
A matrix $S\in\bR^{2d\times 2d}$ is \textbf{symplectic} if $ S^T J S = J $. We denote by $Sp(d,\bR)$ the group of $2d\times2d$ symplectic matrices, which is generated by the symplectic matrices $J$, $\cD_E$ and $V_C$ ($E\in GL(d,\bR)$, $C\in\bR^{d\times d}$, $C$ symmetric), where
\begin{equation}\label{defDEVC}
	\cD_E=\begin{pmatrix} E^{-1} & 0_{d\times d}\\
	0_{d\times d} & E^T\end{pmatrix}, \qquad V_C=\begin{pmatrix} I_{d\times d} & 0_{d\times d} \\ C & I_{d\times d}\end{pmatrix}.
\end{equation}

For $z=(x,\xi)\in\rdd$ and $\tau\in\bR$, we denote with $\rho(z;\tau)=e^{2\pi i\tau}e^{-i\xi\cdot x}\pi(z)$. This is the Schr\"odinger representation of the Heisenberg group. A unitary operator $U:L^2(\rd)\to L^2(\rd)$ which  satisfies the intertwining relationship:
\[
	U^{-1}\rho(z;\tau) U = \rho(Sz;\tau), \qquad z\in\rdd, \  \tau\in\bR,
\]
for a matrix $S\in Sp(d,\bR)$, is called \textbf{metaplectic operator}. The symplectic matrix $S$ is the {projection} of $U$ onto $Sp(d,\bR)$. We write $S=\pi^{Mp}(U)$ and denote $U=\hat S$. Given $S\in Sp(d,\bR)$, the associated metaplectic operator $\hat S$ is unique, up to a sign. The group $\{\pm \hat S : S\in Sp(d,\bR)\}$ of metaplectic operators is denoted by $Mp(d,\bR)$.

For $C\in\bR^{d\times d}$, let us denote 
\begin{equation}\label{chirp}
\Phi_C(t):=e^{i\pi Ct\cdot t}. 
\end{equation}
In the sequel we showcase the main examples of such operators.
\begin{example} $(i)$ The operator $f\in L^2(\rd)\mapsto\Phi_C\cdot f$ is a metaplectic operator and its projection is the symplectic matrix $V_C$ defined in (\ref{defDEVC}).\\
$(ii)$ The normalized rescaling operator $\mathfrak{T}_E: f\in L^2(\rd)\mapsto |\det(E)|^{1/2}f(E\cdot)\in L^2(\rd)$ is a metaplectic operator and $\pi^{Mp}(\mathfrak{T}_E)=\cD_E$, defined as in (\ref{defDEVC}).\\
$(iii)$ The Fourier transform $\cF:f\in L^2(\rd)\mapsto \hat f\in L^2(\rd)$ is a metaplectic operator with $\pi^{Mp}(\cF)=J$.\\
$(iv)$ The partial Fourier transform $\cF_2: f\in L^2(\rdd)\mapsto \cF_2 f\in L^2(\rdd)$ is a metaplectic operator and its projection onto the symplectic group is the symplectic matrix of $Sp(2d,\bR)$ whose block decomposition is
\begin{equation}\label{defAFT2}
	\cA_{FT2}=\begin{pmatrix}
		I_{d\times d}  & 0_{d\times d} & 0_{d\times d} & 0_{d\times d} \\
		0_{d\times d} & 0_{d\times d} & 0_{d\times d} & I_{d\times d} \\
		0_{d\times d} & 0_{d\times d} & I_{d\times d}  & 0_{d\times d} \\
		0_{d\times d} & -I_{d\times d} & 0_{d\times d} & 0_{d\times d}
	\end{pmatrix}.
\end{equation}
\end{example}

\section{Metaplectic Wigner distributions and Atoms}
\begin{definition}
	Let $\hat\cA\in Mp(2d,\bR)$. The \textbf{metaplectic Wigner distribution} associated to $\hat\cA$ is the time-frequency representation 
	\begin{equation}\label{defWA}
		W_\cA(f,g)=\hat\cA(f\otimes\bar g),\quad f,g\in L^2(\rd).
	\end{equation}
\end{definition}

\begin{example}
Examples of metaplectic Wigner distributions are the STFT and the $\tau$-Wigner distributions ($\tau\in\bR$). In fact, for every $f,g\in L^2(\rd)$,
\[
	V_gf(x,\xi)=\cF_2\mathfrak{T}_{ST}(f\otimes\bar g)(x,\xi),\quad \phas\in\rdd
\]
where $\mathfrak{T}_{ST}F(x,y)=F(y, y-x)$ and
\[
	W_\tau(f,g)(x,\xi)=\cF_2\mathfrak{T}_\tau (f\otimes \bar g)(x,\xi),\quad \phas\in\rdd
\]
where 
$
	\mathfrak{T}_\tau F(x,y)=F(x+\tau y,x-(1-\tau)y).
$
\end{example}

We recall the following continuity properties \cite{CR2022p1}.

\begin{proposition} Let $W_\cA$ be a metaplectic Wigner distribution. Then,\\
(i) $W_\cA:L^2(\rd)\times L^2(\rd)\to L^2(\rdd)$ is continuous and Moyal's identity holds:
\begin{equation}\label{moyal}
	\la W_\cA(f,g),W_\cA(\varphi,\psi)\ra =\la f,\varphi\ra \overline{\la g,\psi\ra}, \qquad f,g,\varphi,\psi\in L^2(\rd).
\end{equation}
(ii) $W_\cA:\cS(\rd)\times\cS(\rd)\to \cS(\rdd)$ is continuous;\\
(iii) $W_\cA:\cS'(\rd)\times\cS'(\rd)\to\cS'(\rdd)$ is continuous.
\end{proposition}

An equivalent of the inversion formula (\ref{invSTFT}) for the STFT holds for general metaplectic Wigner distributions, as it was proved in \cite{CG2023frames}, as long as time-frequency shifts are replaced by \textit{metaplectic atoms}.

\begin{definition}
	Let $W_\cA$ be a metaplectic Wigner distribution and $z\in\rdd$. The \textbf{metaplectic atom} $\pi_\cA(z):\cS(\rd)\to \cS'(\rd)$ is the operator defined for all $f\in\cS(\rd)$ by its action on $\cS(\rd)$ as:
	\[
		\la \pi_\cA(z)f,\varphi\ra = \overline{W_\cA(\varphi,f)(z)}, \qquad \varphi\in\cS(\rd).
	\]
\end{definition}

Their main properties were investigated in \cite{CG2023frames}; in particular, we recall the issues below.

For every metaplectic Wigner distribution $W_\cA$ and $z\in\rdd$, $\pi_\cA(z):\cS(\rd)\to\cS'(\rd)$ is bounded.

\begin{theorem}
	Let $W_\cA$ be a metaplectic Wigner distribution. Then, for every $f\in\cS'(\rd)$ and every $g,\gamma\in\cS(\rd)$ such that $\la g,\gamma\ra\neq0$, we have:
	\[
		f = \frac{1}{\la \gamma,g\ra}\int_{\rdd}W_\cA(f,g)(z)\pi_\cA(z)\gamma dz,
	\]
	when the integral must be intended in the weak sense of vector-valued integration.
\end{theorem}

Let $\hat\cA\in Mp(2d,\bR)$ have projection $\cA=\pi^{Mp}(\hat\cA)\in Sp(2d,\bR)$. We write $\cA$ in terms of $d\times d$ blocks as
\begin{equation}\label{blockDecA}
	\cA=\begin{pmatrix}
		A_{11} & A_{12} & A_{13} & A_{14}\\
		A_{21} & A_{22} & A_{23} & A_{24}\\
		A_{31} & A_{32} & A_{33} & A_{34}\\
		A_{41} & A_{42} & A_{43} & A_{44}
	\end{pmatrix}.
\end{equation}
We call $E_\cA$  the submatrix:
\begin{equation}
	E_\cA = \begin{pmatrix}
		A_{11} & A_{13}\\
		A_{21} & A_{23}
	\end{pmatrix}
\end{equation}
and observe that, if $W_\cA$ is the metaplectic Wigner distribution associated to $\hat\cA$,
\[
	|W_\cA(\pi(w)f,g)|=|W_\cA(f,g)(\cdot-E_\cA w)|,
\] 
 for every $f,g\in L^2(\rd)$ and every $w\in\rdd$.
\begin{definition}
	Under the notation above, $W_\cA$ is \textbf{shift-invertible} if $E_\cA\in GL(2d,\bR)$.
\end{definition}

Shift-invertible Wigner distributions are fundamental, as they characterize modulation spaces. Whence, they can be used to measure the local time-frequency content of signals:

\begin{theorem}\label{thmSI}
	Let $W_\cA$ be a shift-invertible metaplectic Wigner distribution. Let $0<p,q\leq\infty$ and $m\in\cM_v(\rdd)$ be such that $m\circ E_\cA^{-1}\asymp m$. Let $g\in\cS(\rd)\setminus\{0\}$. Then,\\
	(i) if $A_{21}=0_{d\times d}$, i.e., if $E_\cA$ is upper-triangular, then $\norm{f}_{M^{p,q}_m}\asymp \norm{W_\cA(f,g)}_{L^{p,q}_m}$ for all $f\in M^{p,q}_m(\rd)$. If $1\leq p,q\leq\infty$, $g$ can be taken in $M^1_v(\rd)$.\\
	(ii) if $0<p=q\leq\infty$, then $\norm{f}_{M^p_m}\asymp\norm{W_\cA(f,g)}_{L^p_m}$ for every $f\in M^p_m(\rd)$.
\end{theorem}
It was showed in \cite{CGshiftinvertible} that if $W_\cA$ is not shift-invertible or $E_\cA$ is not upper-triangular, Theorem \ref{thmSI} fails in general.

\section{Wigner-decomposable atoms}\label{sec:MAWD}
Among all shift-invertible Wigner distributions, the Wigner-decomposable ones, introduced in \cite{CGR2022}, provide a direct generalization of the STFT, as they can be written as $\cF_2\mathfrak{T}_E$ for some $E\in GL(2d,\bR)$, up to a chirp function. In this section, we compute metaplectic atoms associated to Wigner-decomposable metaplectic Wigner distributions.

\begin{definition}
	We say that a matrix $\mathcal{A}\in Sp(d,\mathbb{R})$ is \textbf{Wigner-decomposable} (equivalently, $W_\cA$ is Wigner-decomposable) if $\mathcal{A}=V_C\mathcal{A}_{FT2}\mathcal{D}_E$,  with $\mathcal{D}_E$ and $V_C$ defined in \eqref{defDEVC}, for some $C\in\rd$ symmetric and $E\in GL(d,\mathbb{R})$.  In particular, if $C=0_{d\times d}$ then $V_C=I_{2d\times 2d}$ and we call $\mathcal{A}_{FT2}\mathcal{D}_E$ (or, equivalently, $W_{\mathcal{A}_{FT2}\mathcal{D}_E}$)  \textbf{totally Wigner-decomposable}. 
\end{definition}

Metaplectic atoms of Wigner-decomposable distributions can be expressed in terms of the matrices $C$ and $E$. From now on, we assume that the matrix $E$ enjoys the block decomposition
\begin{equation}\label{matrix-L}
E=\begin{pmatrix}
	E_{11} & E_{12}\\
	E_{21} & E_{22}
\end{pmatrix}.
\end{equation}
In \cite{CGR2022}, the authors prove that a Wigner-decomposable distribution $W_\cA$ is shift-invertible if and only if $E$ is right-regular \footnote{\begin{definition}
		The $2d$-block matrix $E$ in \eqref{matrix-L} is called left-regular (resp. right-regular) if the submatrices $E_{11},E_{21}\in\mathbb{R}^{d\times d}$
		(resp. $E_{12},E_{22}\in\mathbb{R}^{d\times d}$) are invertible. 
\end{definition}}.

\begin{proposition}\label{piAdec}
Assume that	$\mathcal{A}\in Sp(2d,\mathbb{R})$  is totally Wigner-decomposable, that is $\cA=\mathcal{A}_{FT2}\mathcal{D}_E$, with  $E$ in \eqref{matrix-L} right-regular. Then the metapletic atom $\pi_{\cA}\phas$ can be explicitly computed as
\begin{equation}\label{atom-W-dec}
	\pi_{\cA}\phas=\alpha_Ee^{-2\pi i E_{11}E_{12}^{-T}\xi \cdot x}  M_{E_{22}^{-T}\xi}T_{E_{12}E_{22}^{-1}(E_{11}-E_{12}E_{22}^{-1}E_{21})x}\mathfrak{T}_{E_{22}E_{12}^{-1}},
\end{equation}
with $\alpha_E=|\det E|^{1/2}\,|\det(E_{22}E_{12})|^{-1/2}$.
\end{proposition}
\begin{proof}
	By definition, $W_\cA(f,g)\phas=\la f,\pi_\cA\phas g\ra$, for every $f,g\in\cS(\rd)$. In our case, $$W_\cA(f,g)\phas=\cF_2 \mathfrak{T}_{E}(f\otimes\bar{g})\phas=\sqrt{|\det E|}\cF_2(f\otimes\bar{g})(E_{11}x+E_{12}\xi,E_{21}x+E_{22}\xi),$$
	with $E$ in \eqref{matrix-L} having the sub-blocks $E_{12}, E_{22}$ invertible. Using the change of variables $E_{11}x+E_{12}t=s$ so that $dt=|\det E_{12}|^{-1} ds$, and $t=E^{-1}_{12}s-E_{12}^{-1}E_{11}x$, 
	\begin{align*}
		W_\cA(f,g)\phas&=\sqrt{|\det E|}|\det E_{12}|^{-1}\intrd e^{-2\pi i (\xi\cdot E_{12}^{-1}s-\xi\cdot E_{12}^{-1}E_{11}x)}f(s)\\
			&\qquad\qquad\qquad\qquad\qquad\times\,\,\overline{g(E_{21}x+E_{22}E_{12}^{-1}s-E_{22}E_{12}^{-1}E_{11}x)}\,ds\\
		&=\sqrt{|\det E|}|\det E_{12}|^{-1}e^{2\pi i \xi\cdot E_{12}^{-1}E_{11}x}\intrd e^{-2\pi i \xi\cdot E_{12}^{-1}s}f(s)\\
		&\qquad\qquad\qquad\qquad\qquad\times\,\,\overline{g(E_{22}E_{12}^{-1}[s-E_{12}E_{22}^{-1}(E_{11}-E_{12}E_{22}^{-1}E_{21})x])}\,ds\\
		&=\sqrt{|\det E|}|\det E_{12}|^{-1/2}|\det E_{22}|^{-1/2}e^{2\pi i E_{11}E_{12}^{-T}\xi\cdot x}\\
		&\qquad\qquad\qquad\qquad\qquad\times\,\,\la f, M_{E_{22}^{-T}\xi}T_{E_{12}E_{22}^{-1}(E_{11}-E_{12}E_{22}^{-1}E_{21})x}\mathfrak{T}_{E_{22}E_{12}^{-1}}g\ra,
	\end{align*}
as desired.
\end{proof}
\begin{remark}
	Let $W_\cA$ be a Wigner-decomposable distribution with associated symplectic matrix $\cA=V_C\cA_{FT2}\cD_E$, $C\in\bR^{2d\times2d}$ symmetric, $E\in GL(2d,\bR)$ invertible and $\cA_{FT2}$ defined as in (\ref{defAFT2}). The requirement of $E$ to be right-regular is equivalent to $\cA$ being shift-invertible, as shown in \cite[Corollary 4.12]{CGR2022}. Moreover, it follows by \cite[Remark 4.7]{CGR2022} that if $E$ has block decomposition (\ref{matrix-L}), then the matrix $E_\cA$ associated to $\cA$ has block decomposition:
	\[
		E_\cA=\begin{pmatrix}
			(E_{11}-E_{12}E_{22}^{-1}E_{21})^{-1} & 0_{d\times d}\\
			0_{d\times d} & E_{12}^T
		\end{pmatrix}.
	\]
\end{remark}
\begin{corollary}
	Under the assumptions of Proposition \ref{piAdec}, the inverse $\pi^{-1}_\cA$ can be explicitly computed as
	\begin{equation}\label{pi-A-dec-inv}
			\pi_{\cA}^{-1}\phas=\tilde{\alpha}_Ee^{2\pi i E_{11}E_{12}^{-T}\xi \cdot x} \mathfrak{T}_{E_{12}E_{22}^{-1}} T_{-E_{12}E_{22}^{-1}(E_{11}-E_{12}E_{22}^{-1}E_{21})x}M_{-E_{22}^{-T}\xi},\quad x,\xi\in\rd,
	\end{equation}
with $\tilde{\alpha}_E=|(\det E)(\det E_{12}^{-1})(\det E_{22}^{-1})|^{-1/2}$.
\end{corollary}
\begin{proof}
We use $\mathfrak{T}_{E}^{-1}=\mathfrak{T}_{E^{-1}}$, $T_x^{-1}=T_{-x}$, $M_\xi^{-1}=M_{-\xi}$.
\end{proof}
\begin{proposition}
	Assume that	$\mathcal{B}\in Sp(2d,\mathbb{R})$  is Wigner-decomposable, that is $\mathcal{B}=V_C\mathcal{A}_{FT2}\mathcal{D}_E$, with  $E$ in \eqref{matrix-L} right-regular. Then 
	\begin{equation}\label{wig-dec-atom}
		\pi_{\mathcal{B}}=\Phi_{-C}\pi_{\mathcal{A}},
	\end{equation}
where $\Phi_{-C}$ is the chirp function defined in \eqref{chirp} (with $-C$ instead of $C$), and $\pi_\cA$ is the metaplectic atom in \eqref{atom-W-dec}.
\end{proposition}
\begin{proof}
	The proof  uses that $\widehat{V_C}$ is the multiplication operator by the chirp $\Phi_C$. The rest is a straightforward computation. 
\end{proof}

As a consequence, the computation of the inverse $	\pi_{\mathcal{B}}^{-1}$ is immediate:
\begin{corollary}
Under the assumptions of the previous proposition, 
\begin{equation}\label{pi-B-inv}
		\pi_{\mathcal{B}}^{-1}=\Phi_{C}\pi_{\mathcal{A}}^{-1},
\end{equation}
with $\pi_{\mathcal{A}}^{-1}$ computed in \eqref{pi-A-dec-inv}.
\end{corollary}

\section{Wigner-decomposable Gabor frames}
In this section we focus on frame theory e exhibit new Gabor-type frames, which originate from Wigner decomposable distributions.
\begin{definition}
	Let $W_\cA$ be a metaplectic Wigner distribution, $g\in L^2(\rd)\setminus\{0\}$ and $\Lambda\subseteq\rdd$ be a discrete set. A \textbf{metaplectic Gabor frame} of $L^2(\rd)$ is a family $\cG_\cA(g,\Lambda)=\{\pi_\cA(\lambda)g\}_{\lambda\in\Lambda}$ that enjoys the following properties: \\
	(i) $\pi_\cA(\lambda)g\in L^2(\rd)$ for all $\lambda\in\Lambda$;\\
	(ii) there exist $A,B>0$ such that
	\[
		A\norm{f}_2^2\leq\sum_{\lambda\in\Lambda}|W_\cA(f,g)(\lambda)|^2\leq B\norm{f}_2^2, \qquad \forall f\in L^2(\rd).
	\]
\end{definition}

The metaplectic Gabor frames related to the STFT are simply referred to \textit{Gabor frames}.

\begin{theorem}\label{main}
	Let $E\in GL(2d,\bR)$ be right-regular and $W_\cA=\cF_2\mathfrak{T}_E$ be the associated totally  Wigner decomposable distribution. Let $g\in L^2(\rd)\setminus\{0\}$ and $\Lambda\subseteq\rdd$ be a discrete set. The following statements are equivalent:\\
	(i) $\cG_\cA(g,\Lambda)$ is a metaplectic Gabor frame with bounds $A$ and $B$;\\
	(ii) $\cG(\mathfrak{T}_{E_{22}E_{12}^{-1}}g,\cE\Lambda)$ is a Gabor frame with bounds $$\frac{|\det(E_{12}E_{22})|}{|\det(E)|}A\quad \mbox{ \,\,\,\, and\,\,\,\,\,\,\,} \frac{|\det(E_{12}E_{22})|}{|\det(E)|}B,$$
	with $\cE$ being the diagonal matrix
	\[
		\cE :=\begin{pmatrix}
			E_{12}E_{22}^{-1}(E_{11}-E_{12}E_{22}^{-1}E_{21}) & 0_{d\times d}\\
			0_{d\times d} & E_{22}^{-T}
		\end{pmatrix}.
	\]
\end{theorem}
\begin{proof}
	Using (\ref{atom-W-dec}), for all $f\in L^2(\rd)$:
	\begin{align*}
		\sum_{\lambda\in\Lambda}&|W_\cA(f,g)(\lambda)|^2=\sum_{\lambda\in\Lambda}|\la f,\pi_\cA(\lambda)g \ra|^2\\
		&=\frac{|\det(E)|}{|\det(E_{12}E_{22})|}\sum_{(\lambda_1,\lambda_2)\in\Lambda}|\la f, M_{E_{22}^{-T}\lambda_2}T_{E_{12}E_{22}^{-1}(E_{11}-E_{12}E_{22}^{-1}E_{21})\lambda_1} \mathfrak{T}_{E_{22}E_{12}^{-1}}g\ra|^2\\
		&=\frac{|\det(E)|}{|\det(E_{12}E_{22})|}\sum_{(\lambda_1,\lambda_2)\in\Lambda}|\la f, \pi(E_{12}E_{22}^{-1}(E_{11}-E_{12}E_{22}^{-1}E_{21})\lambda_1,E_{22}^{-T}\lambda_2)\mathfrak{T}_{E_{22}E_{12}^{-1}}g \ra|^2\\
		&=\frac{|\det(E)|}{|\det(E_{12}E_{22})|}\sum_{\lambda\in\Lambda}|\la f, \pi(\cE \lambda)\mathfrak{T}_{E_{22}E_{12}^{-1}}g \ra|^2\\
		&=\frac{|\det(E)|}{|\det(E_{12}E_{22})|}\sum_{\mu\in\cE\Lambda}|\la f, \pi(\mu)\mathfrak{T}_{E_{22}E_{12}^{-1}}g \ra|^2,
	\end{align*}
	{where the matrix $\cE$ is invertible because $E$ is invertible, right-regular and $E_{11}-E_{12}E_{22}^{-1}E_{21}$ is its Schur's complement.}
\end{proof}

We mention the characterization of modulation spaces through shift-invertible Wigner-decomposable distributions.

\begin{theorem}\label{main2}
	Let $\hat\cA\in Mp(2d,\bR)$ and $\cA=\pi^{Mp}(\hat\cA)$ be such that the related metaplectic Wigner distribution $W_\cA$ is shift-invertible and Wigner-decomposable, with $\cA=V_C\cA_{FT2}\cD_E$, $C\in\bR^{2d\times2d}$ symmetric, $E\in GL(2d,\bR)$ right-regular with block decomposition (\ref{matrix-L}) and $\cA_{FT2}$ defined as in (\ref{defAFT2}). Let $m\in\cM_v(\rdd)$ satisfy
	\begin{equation}\label{weightcondm}
		m((E_{11}-E_{12}E_{22}^{-1}E_{21})\cdot,E_{12}^{-T}\cdot)\asymp m(\cdot,\cdot)
	\end{equation}
	and $0<p,q\leq\infty$. Then, \\
	(i) for any fixed $g\in\cS(\rd)\setminus\{0\}$,
	\[
		\norm{f}_{M^{p,q}_m}\asymp \norm{W_\cA(f,g)}_{L^{p,q}_m}, \qquad f\in \cS'(\rd).
	\]
	(ii) If $1\leq p,q\leq\infty$, the window $g$ can be chosen in the larger class $M^1_v(\rd)$.
\end{theorem}
\begin{proof}
	Let (\ref{blockDecA}) be the block decomposition of $\cA$. The expression of $W_\cA(f,g)$ in terms of the STFT and the blocks of $\cA$ is:
	\[
		W_\cA(f,g)(x,\xi)=\sqrt{|\det(E)|}|\det(A_{23})|^{-1}e^{2\pi i A_{23}^{-1}\xi\cdot A_{33}^Tx}V_{\tilde g}f(c(x),d(\xi)),
	\]
	where $\tilde g(t)=g(A_{24}^TA_{23}^{-T}t)$, $c(x)=(A_{33}^T-A_{23}^TA_{24}^{-T}A_{34}^T)x$ and $d(\xi)=A_{23}^{-1}\xi$. Using \cite[Remark 4.7]{CGR2022}, we can replace all the blocks of $\cA$ in this expression of $W_\cA$ with those of $E_\cA$:
	\[
		W_\cA(f,g)(x,\xi)=\sqrt{|\det(E)|}|\det(E_{12})|^{-1}e^{2\pi i(E_{12})^{-T}\xi\cdot E_{11}x}V_{\tilde g}f(c(x),d(\xi)),
	\]
	where $\tilde g(t)=g(E_{22}E_{12}^{-1}t)$, $c(x)=(E_{11}-E_{12}E_{22}^{-1}E_{21})x$ and $d(\xi)=E_{12}^{-T}\xi$. We refer to \cite[Theorem 4.9]{CGR2022} for the remaining details.
\end{proof}

Let $\cG_\cA(g,\Lambda)$ be a metaplectic Gabor frame related to the shift-invertible Wigner-decomposable distribution $W_\cA$. The associated \textit{coefficient operator} $C_\cA:L^2(\rd)\to \ell^2(\Lambda)$ is defined as
\[
	C_\cA f(\lambda):=W_\cA(f,g)(\lambda), \qquad f\in L^2(\rd),
\]
whereas its adjoint $D_\cA: \ell^2(\Lambda)\to L^2(\rd)$, called the \textit{reconstruction operator}, is given by 
\[
	D_\cA c=\sum_{\lambda\in\Lambda}c_\lambda\pi_\cA(\lambda)g, \qquad c=(c_\lambda)_{\lambda\in\Lambda}\in\ell^2(\Lambda).
\]
Their composition defines the so-called \textit{frame operator} $S_\cA=D_\cA C_\cA: L^2(\rd)\to L^2(\rd)$, namely:
\[	
	S_\cA f=\sum_{\lambda\in\Lambda}W_\cA(f,g)(\lambda)\pi_\cA(\lambda)g, \qquad f\in L^2(\rd).
\]

\begin{theorem}
	Let $W_\cA$ be as in Theorem \ref{main2}, and $\cG_\cA(g,\Lambda)$, $C_\cA$, $D_\cA$ and $S_\cA$ be defined as above. Let $0<p,q\leq\infty$ and $m\in \cM_v(\rdd)$ satisfying (\ref{weightcondm}). Then,\\
	(i) for every $0<p,q\leq\infty$, $C_\cA: M^{p,q}_m(\rd)\to\ell^{p,q}_m(\Lambda)$ and $D_\cA : \ell^{p,q}_m(\Lambda)\to M^{p,q}_m(\rd)$ continuously. Moreover, if $f\in M^{p,q}_m(\rd)$ and $\gamma_\cA:=\mathfrak{T}_{E_{22}E_{12}^{-1}}S_\cA^{-1}\mathfrak{T}_{E_{12}E_{22}^{-1}}g$, then
	\[
		f=\sum_{\lambda\in\Lambda}W_\cA(f,g)(\lambda)\pi_\cA(\lambda)\gamma_\cA, \qquad f\in M^{p,q}_m(\rd),
	\]
	with unconditional convergence in $M^{p,q}_m(\rd)$ ($0<p,q<\infty$) and unconditional weak*-$M^\infty_{1/v}$ convergence ($\max\{p,q\}=\infty$). \\
	$(ii)$ The following (quasi-)norms are equivalent on $M^{p,q}_m(\rd)$:
	\[
		A\|f\|_{M^{p,q}_m(\rd)}\leq \|(W_\cA(f,g))_{\lambda\in\Lambda}\|_{\ell^{p,q}_m(\Lambda)}\leq B \|f\|_{M^{p,q}_m(\rd)},
	\]
	\[ B^{-1}\|f\|_{M^{p,q}_m(\rd)} \leq  \|(W_\cA(f,g))_{\lambda\in\Lambda}\|_{\ell^{p,q}_m(\Lambda)}\leq A^{-1} \|f\|_{M^{p,q}_m(\rd)}.
	\]
\end{theorem}
\begin{proof}
	It is the transposition of \cite[Theorem 7.3]{CGshiftinvertible} for shift-invertible Wigner-decomposable distributions, observing that in this case: $\widehat{\delta_\cA}=\mathfrak{T}_{E_{12}E_{22}^{-1}}$.
\end{proof}

We conclude this section with the characterization of Wiener amalgam spaces $W(\cF L^p_{m_1},L^q_{m_2})(\rd)$.

\begin{theorem}\label{main3}
	Let $W_\cA$ be as in Theorem \ref{main2}. Let $0<p,q\leq\infty$ and $m_1,m_2\in\cM_v(\rd)$ be such that
	\begin{equation}\label{assumptWeight}
		m_1\otimes m_2\asymp (m_1\circ(-E_{12}^{-T}))\otimes(m_2\circ (E_{11}-E_{12}E_{22}^{-1}E_{21})).
	\end{equation}
	Then, for every $g\in\cS(\rd)\setminus\{0\}$, 
	\[
		\norm{f}_{W(\cF L^p_{m_1},L^q_{m_2})}\asymp \left(\int_{\rd}\left(\int_{\rd}|W_\cA(f,g)(x,\xi)|^pm_1(\xi)^pd\xi\right)^{q/p}m_2(x)dx\right)^{1/q},
	\]
	with the obvious adjustments for $\max\{p,q\}=\infty$.
\end{theorem}
\begin{proof}
	It follows by \cite[Theorem 7.2]{CGshiftinvertible}, observing that
	\[
		JE_\cA^{-1} J =\begin{pmatrix}
		-E_{12}^{-T} & 0_{d\times d}\\
		0_{d\times d} & -(E_{11}-E_{12}E_{22}^{-1}E_{21})
		\end{pmatrix}
	\]
	and, consequently, the assumption (\ref{assumptWeight}) embodies both the conditions $m_2(-\cdot)\asymp m_2(\cdot)$ and $m_1\otimes m_2\asymp (m_1\otimes m_2)\circ (JE_\cA^{-1}J)$ of \cite[Theorem 7.2]{CGshiftinvertible}.
\end{proof}

\section*{Acknowledgements}
The authors have been supported by the Gruppo Nazionale per l’Analisi Matematica, la Probabilità e le loro Applicazioni (GNAMPA) of the Istituto Nazionale di Alta Matematica (INdAM), University of Bologna and HES-SO Valais - Wallis School of Engineering. We acknowledge the support of The Sense Innovation and Research Center, a joint venture of the University of Lausanne (UNIL), The Lausanne University Hospital (CHUV), and The University of Applied Sciences of Western Switzerland – Valais/Wallis (HES-SO Valais/Wallis).

\end{document}